\documentclass[12pt]{article}
\usepackage{amsmath,amssymb,amsthm}
\numberwithin{equation}{section}
\usepackage{units}
\usepackage{color}
\usepackage[T1]{fontenc}
\usepackage[utf8]{inputenc}
\usepackage{authblk}
\usepackage{bm}
\usepackage{graphicx}

\usepackage{datetime}

\usepackage[colorlinks=true,
linkcolor=webgreen,
filecolor=webbrown,
citecolor=webgreen]{hyperref}

\definecolor{webgreen}{rgb}{0,.5,0}
\definecolor{webbrown}{rgb}{.6,0,0}

\usepackage[toc,page]{appendix}

\newcommand{\Z}{{\mathbb Z}}

\newtheorem{thm}{Theorem}
\newtheorem{theorem}[thm]{Theorem}

\newtheorem{corollary}[thm]{Corollary}

\DeclareMathOperator{\Li}{Li}

\title{Relating certain weighted Fibonacci series to Bernoulli polynomials via the polylogarithm function}

\author[]{Kunle Adegoke \\\href{mailto:adegoke00@gmail.com}{\tt adegoke00@gmail.com}}

\affil{Department of Physics and Engineering Physics, \mbox{Obafemi Awolowo University}, 220005 Ile-Ife, Nigeria}

\begin{document}
\date{}

\maketitle

\begin{abstract} \noindent
We show that certain weighted Fibonacci and Lucas series can always be expressed as linear combinations of polylogarithms. In some special cases we evaluate the series in terms of Bernoulli polynomials, making use of the connection between these polynomials and the polylogarithm.  
\end{abstract}
\section{Introduction}
This paper is concerned mainly with relating weighted Fibonacci and Lucas series of the form
\[
\sum_{j = 1}^\infty  {\frac{{( - 1)^{j - 1} }}{{j^k }}F_{rj} } \mbox{ and } \sum_{j = 1}^\infty  {\frac{{( - 1)^{j - 1} }}{{j^k }}L_{rj} }\,,
\]
and related series to Bernoulli polynomials through the polylogarithm function. Here $r$ and $k$ are integers and $F_n$ and $L_n$ are Fibonacci and Lucas numbers.

\medskip

Among other results of a similar nature, we will establish that
\[
\sum_{j = 1}^\infty  {\frac{{( - 1)^{j - 1} }}{{j^k }}F_{rj} }  =   \frac{{(2\pi i)^k }}{{k!\sqrt 5 }}B_k \left( {\frac{1}{2} + \frac{{r\log \alpha }}{{2\pi i}}} \right),\quad \mbox{$k$ odd, $k\ge 1$}\,,
\]
and
\[
\sum_{j = 1}^\infty  {\frac{{( - 1)^{j - 1} }}{{j^k }}L_{rj} }  =   \frac{{(2\pi i)^k }}{{k!}}B_k \left( {\frac{1}{2} + \frac{{r\log \alpha }}{{2\pi i}}} \right),\quad \mbox{$k$ even, $k\ge 0$}\,,
\]
where $r$ is an even integer, $\alpha$ is the golden ratio and $B_n(x)$ is the $n^{th}$ Bernoulli polynomial in $x$.

\medskip

The Fibonacci numbers, $F_n$, and the Lucas numbers, $L_n$, are defined, for \mbox{$n\in\Z$}, through the recurrence relations 
\begin{equation}\label{eq.s6z1bcx}
F_n=F_{n-1}+F_{n-2}, \mbox{($n\ge 2$)},\quad\mbox{$F_0=0$, $F_1=1$};
\end{equation}
and
\begin{equation}
L_n=L_{n-1}+L_{n-2}, \mbox{($n\ge 2$)},\quad\mbox{$L_0=2$, $L_1=1$};
\end{equation}
with
\begin{equation}
F_{-n}=(-1)^{n-1}F_n\,,\quad L_{-n}=(-1)^nL_n\,.
\end{equation}
Throughout this paper, we denote the golden ratio, $(1+\sqrt 5)/2$, by $\alpha$ and write $\beta=(1-\sqrt 5)/2=-1/\alpha$, so that $\alpha\beta=-1$ and $\alpha+\beta=1$. 

\medskip

Explicit formulas (Binet formulas) for the Fibonacci and Lucas numbers are
\begin{equation}
F_n  = \frac{{\alpha ^n  - \beta ^n }}{{\alpha  - \beta }},\quad L_n  = \alpha ^n  + \beta ^n,\quad n\in\Z\,.
\end{equation}
Koshy \cite{koshy} and Vajda \cite{vajda} have written excellent books dealing with Fibonacci and Lucas numbers.

\medskip

As will be shown in Theorem \ref{thm.main}, sums of the form
\[
\sum_{j = 1}^\infty  {\frac{{z^{j - 1} }}{{j^k }}F_{rj+s} } \mbox{ and } \sum_{j = 1}^\infty  {\frac{{z^{j - 1} }}{{j^k }}L_{rj+s} }\,,
\]
can always be expressed as linear combinations of polylogarithms. For $z=\pm 1$, $s=0$ and a definite parity of (non-negative) $k$, the series are expressible in terms of the Bernoulli polynomials, for even $r$ (since $\alpha^r=1/\beta^r=t$ for even $r$), in view of identities \eqref{eq.bernpoly1} and \eqref{eq.bernpoly2}. For a general $z$ and arbitrary $k$, the series can be evaluated in terms of elementary functions whenever the same is true about the corresponding linear combination of polylogarithms, as demonstrated in the following examples, presented in the later part of section~\ref{sec.others}:
\begin{itemize}
\item Theorem \ref{thm.xzk5qh8} on page \pageref{thm.xzk5qh8}:
\[
\sum_{j = 1}^\infty  {\frac{{( - 1)^{j-1} }}{{j^2 }}F_{j + s} }  = F_s \log ^2 \alpha  + \frac{{\pi ^2 }}{{50}}L_s \sqrt 5,\quad s\in\Z\,;
\]
\item Theorem \ref{thm.vpcaqaj} on page \pageref{thm.vpcaqaj}:
\[
\sum_{j = 1}^\infty  {\frac{{L_j }}{{2^j j^2 }}}  = \frac{{\pi ^2 }}{{12}} + 2\log ^2 \alpha  - \log ^2 2\,;
\]
\item Theorem \ref{thm.j7p1nrz} on page \pageref{thm.j7p1nrz}:
\[
\sum_{j = 1}^\infty  {\frac{{L_{rj} }}{{L_r^j j^2 }}}  = \frac{{\pi ^2 }}{6} - \log (\alpha ^r /L_r )\log (\beta ^r /L_r ),\quad \mbox{$r$ even}\,;
\]
\item Theorem \ref{thm.ejojlyt} on page \pageref{thm.ejojlyt}:
\[
\mbox{ }\sum_{j = 1}^\infty  {\frac{{( - 1)^{j - 1} }}{{j^3 }}L_j }  = \frac{1}{5}\left( {\pi ^2 \log \alpha  - \zeta (3)} \right)\,.
\]
\end{itemize}
For real or complex order $k$ and argument $z$, the polylogarithm $\Li_k(z)$ is defined by
\begin{equation}\label{eq.zz0im7f}
\Li_k (z) = \sum_{j = 1}^\infty  {\frac{{z^j }}{{j^k }}},\quad|z|<1\,.
\end{equation}
The series \eqref{eq.zz0im7f} also converges when $|z|=1$, provided that $\Re k>1$. For $|z|>1$, $\Li_k(z)$ is defined by analytic continuation.
The special case $k=1$ involves the natural logarithm, $\Li_1(z)=-\log(1-z)$, while $k=2$ and $k=3$ are called the dilogarithm and the trilogarithm. For non-positive integer orders $k$, the polylogarithm is a rational function. In fact,
\begin{equation}
\Li_0 (z) = \frac{z}{{1 - z}}\,,
\end{equation}

\begin{equation}
\Li_{ - 1} (z) = \frac{z}{{(1 - z)^2 }}\,,
\end{equation}

\begin{equation}
\Li_{ - 2} (z) = \frac{{z(1 + z)}}{{(1 - z)^3 }}\,,
\end{equation}
and more generally,
\begin{equation}
\Li_{ - n} (z) = \left( {z\frac{d}{{dz}}} \right)^n \frac{z}{{1 - z}}\,.
\end{equation}
The book by Lewin \cite{lewin81} is a rich source of information on the polylogarithm function.

\medskip

The Bernoulli numbers, $B_k$, are defined by the generating function
\begin{equation}
\frac{z}{{e^z  - 1}} = \sum_{k = 0}^\infty  {B_k \frac{{z^k }}{{k!}}},\quad z<2\pi\,,
\end{equation}
and the Bernoulli polynomials by the generating function
\begin{equation}
\frac{{ze^{xz} }}{{e^z  - 1}} = \sum_{k = 0}^\infty  {B_k (x)\frac{{z^k }}{{k!}}},\quad|z|<2\pi\,.
\end{equation}
Clearly, $B_k=B_k(0)$.

\medskip

The first few Bernoulli numbers are
\begin{equation}
B_0  = 1,\,B_1  =  - \frac{1}{2},\,B_2  = \frac{1}{6},\,B_3  = 0,\,B_4  =  - \frac{1}{{30}},\,B_5  = 0,\,B_6  = \frac{1}{{42}},\,B_7  = 0,\,\ldots\,,
\end{equation}
while the first few Bernoulli polynomials are
\begin{equation}
\begin{split}
&B_0 (x)= 1,\quad B_1 (x) = x - \frac{1}{2},\quad B_2 (x) = x^2  - x + \frac{1}{6},\quad B_3 (x) = x^3  - \frac{3}{2}x^2  + \frac{1}{2}x\,,\\
&B_4 (x) = x^4  - 2x^3  + x^2  - \frac{1}{{30}},\quad B_5 (x) = x^5  - \frac{5}{2}x^4  + \frac{5}{3}x^3  - \frac{1}{6}x\,,\\
& B_6 (x) = x^6  - 3x^5  + \frac{5}{2}x^4  - \frac{1}{2}x^2  + \frac{1}{{42}}\,.
\end{split}
\end{equation}
An explicit formula for the Bernoulli polynomials is
\begin{equation}
B_k (x) = \sum_{j = 0}^k {\binom kjB_j x^{k - j} }\,,
\end{equation}
while a recurrence formula for them is
\begin{equation}
B_k (x + 1) = \sum_{j = 0}^k {\binom kjB_j (x)}\,.
\end{equation}
Some of the main results in this paper derive from the following relationship between the polylogarithm and the Bernoulli polynomials (see Lewin~\cite[equation 7.192]{lewin81}, where we have corrected a misprint):
\begin{equation}\label{eq.bernpoly1}
\Li_k (t) + ( - 1)^k \Li_k (1/t) =  - \frac{{(2\pi i)^k }}{{k!}}B_k \left( {\frac{{\log t}}{{2\pi i}}} \right),\quad\mbox{$t\le1$}\,,
\end{equation}
and
\begin{equation}\label{eq.bernpoly2}
\Li_k (t) + ( - 1)^k \Li_k (1/t) =  - \frac{{(2\pi i)^k }}{{k!}}B_k \left( { - \frac{{\log t}}{{2\pi i}}} \right),\quad t>1\,.
\end{equation}
Basic properties of the Bernoulli polynomials are highlighted in recent articles by Frontczak~\cite{frontczak19} and by Frontczak and Goy~\cite{frontczak20} where new identities involving Fibonacci and Bernoulli numbers, and Lucas and Euler numbers are presented. Additional information on Bernoulli polynomials can be found in Erd\'elyi et al \cite[\S1.13]{erdelyi}. 
\section{Main results}
\begin{theorem}\label{thm.main}
Let $r$, $k$ and $s$ be integers and $z$ a real or complex variable such that $|z|<\alpha^{-r}$. Then,
\begin{equation}\label{eq.h9ujlfl}
\begin{split}
\sum_{j = 1}^\infty  {\frac{{z^j }}{{j^k }} F_{rj + s}}  &= \frac{{F_s }}{2}\left( {\Li_k (\alpha ^r z) + \Li_k (\beta ^r z)} \right)\\
&\qquad+ \frac{{L_s }}{{2\sqrt 5 }}\left( {\Li_k (\alpha ^r z) - \Li_k (\beta ^r z)} \right)\,,
\end{split}
\end{equation}

\begin{equation}\label{eq.nr41hks}
\begin{split}
\sum_{j = 1}^\infty  {\frac{{z^j }}{{j^k }} L_{rj + s}}  &= \frac{{L_s }}{2}\left( {\Li_k (\alpha ^r z) + \Li_k (\beta ^r z)} \right)\\
&\qquad+ \frac{{F_s \sqrt 5 }}{2}\left( {\Li_k (\alpha ^r z) - \Li_k (\beta ^r z)} \right)\,.
\end{split}
\end{equation}
\end{theorem}
\begin{proof}
We have
\begin{equation}\label{eq.irdpf3d}
\alpha ^s \Li_k (\alpha ^r z) = \sum_{j = 1}^\infty  {\frac{{z^j }}{{j^k }} \alpha ^{rj + s}}  = \frac{1}{2}\sum_{j = 1}^\infty  {\frac{{z^j }}{{j^k }} L_{rj + s}}  + \frac{{\sqrt 5 }}{2}\sum_{j = 1}^\infty  {\frac{{z^j }}{{j^k }} F_{rj + s}}\,,
\end{equation}
and
\begin{equation}\label{eq.frq62gc}
\beta ^s \Li_k (\beta ^r z) = \sum_{j = 1}^\infty  {\frac{{z^j }}{{j^k }} \beta ^{rj + s}}  = \frac{1}{2}\sum_{j = 1}^\infty  {\frac{{z^j }}{{j^k }} L_{rj + s}}  - \frac{{\sqrt 5 }}{2}\sum_{j = 1}^\infty  {\frac{{z^j }}{{j^k }} F_{rj + s}}\,;
\end{equation}
where we have used the fact that, for any integer $m$,
\begin{equation}\label{eq.we8ccor}
\alpha ^m  = \frac{{L_m  + F_m \sqrt 5 }}{2}\mbox{ and }\beta ^m  = \frac{{L_m  - F_m \sqrt 5 }}{2}\,.
\end{equation}
From \eqref{eq.irdpf3d} and \eqref{eq.frq62gc} we get
\begin{equation}\label{eq.z444ygq}
\sqrt 5 \sum_{j = 1}^\infty  {\frac{{z^j }}{{j^k }} F_{rj + s}}  = \alpha ^s \Li_k (\alpha ^r z) - \beta ^s \Li_k (\beta ^r z)\,,
\end{equation}
and
\begin{equation}\label{eq.x42cf1i}
\sum_{j = 1}^\infty  {\frac{{z^j }}{{j^k }} L_{rj + s}}  = \alpha ^s \Li_k (\alpha ^r z) + \beta ^s \Li_k (\beta ^r z)\,,
\end{equation}
from which identities \eqref{eq.h9ujlfl} and \eqref{eq.nr41hks} follow from the application of identities~\eqref{eq.we8ccor} to resolve $\alpha^s$ and $\beta^s$.
\end{proof}
Setting $s=0$ in \eqref{eq.z444ygq} and \eqref{eq.x42cf1i} we get the particular cases
\begin{equation}\label{eq.partfibo}
\sum_{j = 1}^\infty  {\frac{{z^j }}{{j^k }} F_{rj} }  = \frac{1}{{\sqrt 5 }}\left( {\Li_k (\alpha ^r z) - \Li_k (\beta ^r z)} \right)\,,
\end{equation}
and
\begin{equation}\label{eq.partluca}
\sum_{j = 1}^\infty  {\frac{{z^j }}{{j^k }} L_{rj} }  = \Li_k (\alpha ^r z) + \Li_k (\beta ^r z)\,.
\end{equation}
\begin{theorem}
Let $r$ be an even integer and $k$ a non-negative integer. Then
\begin{equation}\label{eq.aujl7hb}
\sum_{j = 1}^\infty  {\frac{{( - 1)^{j - 1} }}{{j^k }}F_{rj} }  =   \frac{{(2\pi i)^k }}{{k!\sqrt 5 }}B_k \left( {\frac{1}{2} + \frac{{r\log \alpha }}{{2\pi i}}} \right),\quad \mbox{$k$ odd}\,,
\end{equation}

\begin{equation}\label{eq.cq80nxw}
\sum_{j = 1}^\infty  {\frac{{( - 1)^{j - 1} }}{{j^k }}L_{rj} }  =   \frac{{(2\pi i)^k }}{{k!}}B_k \left( {\frac{1}{2} + \frac{{r\log \alpha }}{{2\pi i}}} \right),\quad \mbox{$k$ even}\,.
\end{equation}
\end{theorem}
\begin{proof}
Identity \eqref{eq.partfibo} with $z=-1$ gives
\[
\sum_{j = 1}^\infty  {\frac{{( - 1)^j }}{{j^k }}F_{rj} }  = \frac{1}{{\sqrt 5 }}\left( {\Li_k ( - \alpha ^r ) - \Li_k ( - \beta ^r )} \right)\,,
\]
which, since $\beta^r=1/\alpha^r$ if $r$ is even, means
\[
\sum_{j = 1}^\infty  {\frac{{( - 1)^j }}{{j^k }}F_{rj} }  = \frac{1}{{\sqrt 5 }}\left( {\Li_k ( - \alpha ^r ) - \Li_k ( - 1/\alpha ^r )} \right),\quad\mbox{$r$ even}\,,
\]
from which, if $k$ is an odd non-negative integer, identity \eqref{eq.aujl7hb} follows on account of relation~\eqref{eq.bernpoly1}. The proof of identity~\eqref{eq.cq80nxw} is similar; we use $z=-1$ in identity~\eqref{eq.partluca} and use~\eqref{eq.bernpoly1} with $k$ even.
\end{proof}
\begin{theorem}
Let $r$ be an integer and $k$ a non-negative integer. Then,
\begin{equation}\label{eq.l3evr8n}
\sum_{j = 1}^\infty  {\left( {\frac{{F_{(4j - 2)r} }}{{(4j - 2)^k }} - \frac{{F_{4jr} }}{{(4j)^k }}} \right)}  = \frac{{(2\pi i)^k }}{{k!2^k\sqrt 5 }}B_k \left( {\frac{1}{2} + \frac{{r\log \alpha }}{{\pi i}}} \right),\quad\mbox{$k$ odd}\,,
\end{equation}
\begin{equation}\label{eq.oe1nr9e}
\sum_{j = 1}^\infty  {\left( {\frac{{L_{(4j - 2)r} }}{{(4j - 2)^k }} - \frac{{L_{4jr} }}{{(4j)^k }}} \right)}  = \frac{{(2\pi i)^k }}{{k!2^k}}B_k \left( {\frac{1}{2} + \frac{{r\log \alpha }}{{\pi i}}} \right),\quad\mbox{$k$ even}\,.
\end{equation}
\end{theorem}
\begin{proof}
Setting $z=i$ (the imaginary unit) in identity \eqref{eq.partfibo} and taking real parts, we have
\begin{equation}
\sum_{j = 1}^\infty  {\frac{{\cos (j\pi /2)}}{{j^k }}F_{rj} }  = \frac{1}{{\sqrt 5 }}\left( {\Re \Li_k (\alpha ^r e^{i\pi /2} ) - \Re \Li_k (\beta ^r e^{i\pi /2} )} \right)\,.
\end{equation}
For a real variable $y$, (Lewin \cite[page 300, Formula (33)]{lewin81}) :
\begin{equation}
\Re \Li_k (ye^{i\pi /2} ) = \frac{1}{{2^k }}\Li_k ( - y^2 )\,.
\end{equation}
Thus, we have
\[
\sum_{j = 1}^\infty  {\frac{{\cos (j\pi /2)}}{{j^k }}F_{rj} }  = \frac{1}{{2^k\sqrt 5 }}\left( {\Li_k ( - \alpha ^{2r} ) - \Li_k ( - \beta ^{2r} )} \right)\,,
\]
from which identity~\eqref{eq.l3evr8n} now follows from relation~\eqref{eq.bernpoly1} and the fact that
\[
\sum_{j = 1}^\infty  {f(j)\cos (j\pi /2)}  \equiv \sum_{j = 1}^\infty  {\left( {f(4j) - f(4j - 2)} \right)} \,,
\]
for an arbitrary real sequence $f(j)$. The proof of identity~\eqref{eq.oe1nr9e} is similar; we set $z=i$ in \eqref{eq.partluca} and take real parts.
\end{proof}
\begin{theorem}\label{thm.nqsu925}
Let $k$ be a non-negative even integer. Let $r$ and $s$ be integers having the same parity. Then,
\[
\sum_{j = 1}^\infty  {\frac{{( - 1)^{j - 1} }}{{j^k }}F_{rj} F_{sj} }  = \left\{ \begin{array}{l}
 \frac{{(2\pi i)^k }}{{k!}}\left( {\frac{1}{5}B_k \left( {\frac{1}{2} + \frac{{(s + r)\log \alpha }}{{2\pi i}}} \right) - \frac{1}{5}B_k \left( {\frac{1}{2} + \frac{{(s - r)\log \alpha }}{{2\pi i}}} \right)} \right)\,,\quad \mbox{$r$ even}\,, \\ 
\\
 \frac{{(2\pi i)^k }}{{k!}}\left( {\frac{1}{5}B_k \left( {\frac{1}{2} + \frac{{(s + r)\log \alpha }}{{2\pi i}}} \right) - \frac{1}{5}B_k \left( {\frac{{(s - r)\log \alpha }}{{2\pi i}}} \right)} \right)\,,\quad \mbox{$r$ odd, $s\le r$}\,. \\ 
 \end{array} \right.
\]

\end{theorem}
\begin{proof}
Writing $\alpha^sz$ for $z$ in \eqref{eq.partfibo}, we have
\[
\sum_{j = 1}^\infty  {\frac{{z^j }}{{j^k }}F_{rj} \alpha ^{sj} z^j }  = \frac{1}{{\sqrt 5 }}\left( {\Li_k (\alpha ^{s + r} z) - \Li_k (\beta ^r \alpha ^s z)} \right)\,,
\]
which, using \eqref{eq.we8ccor}, can be written as
\begin{equation}\label{eq.d48b4kn}
\frac{1}{2}\sum_{j = 1}^\infty  {\frac{{z^j }}{{j^k }}F_{rj} L_{sj} }  + \frac{{\sqrt 5 }}{2}\sum_{j = 1}^\infty  {\frac{{z^j }}{{j^k }}F_{rj} F_{sj} }  = \frac{1}{{\sqrt 5 }}\left( {\Li_k (\alpha ^{s + r} z) - \Li_k (\beta ^r \alpha ^s z)} \right)\,.
\end{equation}
Writing $\beta^sz$ for $z$ in \eqref{eq.partfibo} and making use of \eqref{eq.we8ccor}, we have
\begin{equation}\label{eq.qqg4lst}
\frac{1}{2}\sum_{j = 1}^\infty  {\frac{{z^j }}{{j^k }}F_{rj} L_{sj} }  - \frac{{\sqrt 5 }}{2}\sum_{j = 1}^\infty  {\frac{{z^j }}{{j^k }}F_{rj} F_{sj} }  = \frac{1}{{\sqrt 5 }}\left( {\Li_k (\alpha ^r \beta ^s z) - \Li_k (\beta ^{s + r} z)} \right)\,.
\end{equation}
Subtraction of \eqref{eq.qqg4lst} from \eqref{eq.d48b4kn} gives 
\begin{equation}\label{eq.ezhz8t1}
\begin{split}
\sum_{j = 1}^\infty  {\frac{{z^j }}{{j^k }}F_{rj} F_{sj} }  &= \frac{1}{5}\left( {\Li_k (\alpha ^{s + r} z) + \Li_k (\beta ^{s + r} z)} \right)\\
&\qquad - \frac{1}{5}\left( {\Li_k (( - 1)^r \alpha ^{s - r} z) + \Li_k (( - 1)^r \beta ^{s - r} z)} \right)\,,
\end{split}
\end{equation}
while their addition yields
\begin{equation}\label{eq.qii2vmo}
\begin{split}
\sum_{j = 1}^\infty  {\frac{{z^j }}{{j^k }}F_{rj} L_{sj} }  &= \frac{1}{{\sqrt 5 }}\left( {\Li_k (\alpha ^{s + r} z) - \Li_k (\beta ^{s + r} z)} \right)\\
&\qquad - \frac{1}{{\sqrt 5 }}\left( {\Li_k (( - 1)^r \alpha ^{s - r} z) - \Li_k (( - 1)^r \beta ^{s - r} z)} \right)\,.
\end{split}
\end{equation}
\end{proof}
Setting $z=-1$ in \eqref{eq.ezhz8t1} gives the identity of Theorem \ref{thm.nqsu925} in light of relation~\eqref{eq.bernpoly1}.

\medskip

We have the following example evaluations from Theorem \ref{thm.nqsu925}:
\begin{equation}
\sum_{j = 1}^\infty  {\frac{{( - 1)^{j - 1} }}{{j^4 }}F_{2j} F_{4j} }  = \frac{8}{{15}}\pi ^2 \log ^2 \alpha  + \frac{{32}}{3}\log ^4 \alpha\,,
\end{equation}

\begin{equation}
\sum_{j = 1}^\infty  {\frac{{( - 1)^{j - 1} }}{{j^6 }}F_{2j} F_{4j} }  = \frac{{14}}{{225}}\pi ^4 \log ^2 \alpha  + \frac{{16}}{9}\pi ^2 \log ^4 \alpha  + \frac{{2912}}{{225}}\log ^6 \alpha\,,
\end{equation}
\begin{corollary}\label{cor.hem6pmm}
Let $k$ be a non-negative even integer. Then,
\begin{equation}
\sum_{j = 1}^\infty  {\frac{{( - 1)^{j - 1} }}{{j^k }}F_{rj}^2 }  = \left\{ \begin{array}{l}
 \frac{{(2\pi i)^k }}{{k!}}\left( {\frac{1}{5}B_k \left( {\frac{1}{2} + \frac{{r\log \alpha }}{{\pi i}}} \right) - \frac{1}{5}(2^{1 - k}-1)B_k} \right)\,,\quad \mbox{$r$ even}\,, \\ 
\\
 \frac{{(2\pi i)^k }}{{k!}}\left( {\frac{1}{5}B_k \left( {\frac{1}{2} + \frac{{r\log \alpha }}{{\pi i}}} \right) - \frac{1}{5}B_k} \right)\,,\quad \mbox{$r$ odd}\,. \\ 
 \end{array} \right.
\end{equation}
\end{corollary}
Examples from Corollary \ref{cor.hem6pmm} include
\begin{equation}
\sum_{j = 1}^\infty  {\frac{{( - 1)^{j - 1} }}{{j^2 }}F_{_{2j} }^2 }  = \frac{8}{{5}}\log ^2 \alpha\,,
\end{equation}
\begin{equation}
\sum_{j = 1}^\infty  {\frac{{( - 1)^{j - 1} }}{{j^6 }}F_{_{4j} }^2 }  = \frac{4}{{225}}\log ^2 \alpha \left( {7\pi ^4  + 320\pi ^2 \log ^2 \alpha  + 4096\log ^4 \alpha } \right)\,,
\end{equation}
\begin{equation}
\sum_{j = 1}^\infty  {\frac{{( - 1)^{j - 1} }}{{j^6 }}F_j^2 }  = \frac{{\pi ^6 }}{{1200}} + \frac{{7\pi ^4 }}{{900}}\log ^2 \alpha  + \frac{{\pi ^2 }}{{45}}\log ^4 \alpha  + \frac{4}{{225}}\log ^6 \alpha\,,
\end{equation}
\begin{equation}
\sum_{j = 1}^\infty  {\frac{{( - 1)^{j - 1} }}{{j^6 }}F_{3j}^2 }  = \frac{{\pi ^6 }}{{1200}} + \frac{{7\pi ^4 }}{{900}}\log ^2 \alpha  + \frac{{9\pi ^2 }}{5}\log ^4 \alpha  + \frac{{324}}{{25}}\log ^6 \alpha\,.
\end{equation}
\begin{theorem}\label{eq.wnwp0c3}
Let $k$ be a non-negative odd integer. Let $r$ and $s$ be integers having the same parity. Then,
\begin{equation}
\sum_{j = 1}^\infty  {\frac{{( - 1)^{j - 1} }}{{j^k }}F_{rj} L_{sj} }  = \left\{ \begin{array}{l}
 \frac{{(2\pi i)^k }}{{k!}}\left( {\frac{1}{\sqrt 5}B_k \left( {\frac{1}{2} + \frac{{(s + r)\log \alpha }}{{2\pi i}}} \right) - \frac{1}{\sqrt 5}B_k \left( {\frac{1}{2} + \frac{{(s - r)\log \alpha }}{{2\pi i}}} \right)} \right)\,,\quad \mbox{$r$ even}\,, \\ 
\\
 \frac{{(2\pi i)^k }}{{k!}}\left( {\frac{1}{\sqrt 5}B_k \left( {\frac{1}{2} + \frac{{(s + r)\log \alpha }}{{2\pi i}}} \right) - \frac{1}{\sqrt 5}B_k \left( {\frac{{(s - r)\log \alpha }}{{2\pi i}}} \right)} \right)\,,\quad \mbox{$r$ odd, $s\le r$}\,. \\ 
 \end{array} \right.
\end{equation}
\end{theorem}
\begin{proof}
Set $z=-1$ in identity~\eqref{eq.qii2vmo} and make use of relation~\eqref{eq.bernpoly1}.
\end{proof}
Examples from Theorem \eqref{eq.wnwp0c3} include
\begin{equation}
\sum_{j = 1}^\infty  {\frac{{( - 1)^{j - 1} }}{j}F_{2j} L_{4j} }  = \frac{4}{{\sqrt 5 }}\log \alpha\,,
\end{equation}
\begin{equation}
\sum_{j = 1}^\infty  {\frac{{( - 1)^{j - 1} }}{{j^3 }}F_{2j} L_{4j} }  = \frac{2}{{3\sqrt 5 }}(\pi ^2  + 52\log ^2 \alpha )\log \alpha\,.
\end{equation}
\begin{theorem}\label{eq.zyd5el6}
Let $k$ be a non-negative even integer. Let $r$ and $s$ be integers having the same parity. Then,
\begin{equation}
\sum_{j = 1}^\infty  {\frac{{( - 1)^{j - 1} }}{{j^k }}L_{rj} L_{sj} }  = \left\{ \begin{array}{l}
 \frac{{(2\pi i)^k }}{{k!}}\left( {B_k \left( {\frac{1}{2} + \frac{{(s + r)\log \alpha }}{{2\pi i}}} \right) + B_k \left( {\frac{1}{2} + \frac{{(s - r)\log \alpha }}{{2\pi i}}} \right)} \right)\,,\quad \mbox{$r$ even}\,, \\ 
\\
 \frac{{(2\pi i)^k }}{{k!}}\left( {B_k \left( {\frac{1}{2} + \frac{{(s + r)\log \alpha }}{{2\pi i}}} \right) + B_k \left( {\frac{{(s - r)\log \alpha }}{{2\pi i}}} \right)} \right)\,,\quad \mbox{$r$ odd, $s\le r$}\,. \\ 
 \end{array} \right.
\end{equation}
\end{theorem}
\begin{proof}
The proof is similar to that of Theorem~\ref{thm.nqsu925}. We write $\alpha^sz$ and $\beta^sz$ for $z$, in turn, in identity~\eqref{eq.partluca} and make use of identity~\eqref{eq.we8ccor}.
\end{proof}
We have the following examples from Theorem \eqref{eq.zyd5el6}:
\begin{equation}
\sum_{j = 1}^\infty  {\frac{{( - 1)^{j - 1} }}{j^2}L_{2j} L_{4j} }  = \frac{{\pi ^2 }}{3} + 20\log ^2 \alpha\,,
\end{equation}
\begin{equation}
\sum_{j = 1}^\infty  {\frac{{( - 1)^{j - 1} }}{{j^4 }}L_{2j} L_{4j} }  = \frac{{7\pi ^4 }}{{180}} + \frac{{10\pi ^2 }}{3}\log ^2 \alpha  + \frac{{164}}{3}\log ^4 \alpha\,,
\end{equation}
\begin{corollary}\label{cor.s1q2vhc}
Let $k$ be a non-negative even integer. Then,
\begin{equation}
\sum_{j = 1}^\infty  {\frac{{( - 1)^{j - 1} }}{{j^k }}L_{rj}^2 }  = \left\{ \begin{array}{l}
 \frac{{(2\pi i)^k }}{{k!}}\left( {B_k \left( {\frac{1}{2} + \frac{{r\log \alpha }}{{\pi i}}} \right) + (2^{1 - k}-1)B_k} \right)\,,\quad \mbox{$r$ even}\,, \\ 
\\
 \frac{{(2\pi i)^k }}{{k!}}\left( {B_k \left( {\frac{1}{2} + \frac{{r\log \alpha }}{{\pi i}}} \right) + B_k} \right)\,,\quad \mbox{$r$ odd}\,. \\ 
 \end{array} \right.
\end{equation}
\end{corollary}
Examples from Corollary \ref{cor.s1q2vhc} include
\begin{equation}
\sum_{j = 1}^\infty  {\frac{{( - 1)^{j - 1} }}{{j^2 }}L_{2j}^2 }  = \frac{{\pi ^2 }}{3} + 8\log ^2 \alpha\,, 
\end{equation}
\begin{equation}
\sum_{j = 1}^\infty  {\frac{{( - 1)^{j - 1} }}{{j^4 }}L_{2j}^2 }  = \frac{{7\pi ^4 }}{{180}} + \frac{4}{3}\pi ^2 \log ^2 \alpha  + \frac{{32}}{3}\log ^4 \alpha\,, 
\end{equation}
\begin{equation}
\sum_{j = 1}^\infty  {\frac{{( - 1)^{j - 1} }}{{j^6 }}L_j^2 }  =  - \frac{{\pi ^6 }}{{15120}} + \frac{7}{{180}}\pi ^4 \log ^2 \alpha  + \frac{1}{9}\pi ^2 \log ^4 \alpha  + \frac{4}{{45}}\log ^6 \alpha\,. 
\end{equation}
\section{Miscellaneous results}\label{sec.others}
\begin{theorem}
Let $r$ and $s$ be integers. Let $z$ be a real or complex variable for which $|z|<\alpha^{-r}$. Then,
\begin{equation}\label{eq.ruu9ohg}
\sum_{j = 1}^\infty  {z^j F_{rj + s} }  = \frac{{F_{r + s} z - ( - 1)^r z^2 F_s }}{{1 - L_r z + ( - 1)^r z^2 }}\,,
\end{equation}

\begin{equation}\label{eq.g74w6m7}
\sum_{j = 1}^\infty  {z^j L_{rj + s} }  = \frac{{L_{r + s} z - ( - 1)^r z^2 L_s }}{{1 - L_r z + ( - 1)^r z^2 }}\,.
\end{equation}

\end{theorem}
\begin{proof}
Use
\[
\Li_0 (\alpha ^r z) + \Li_0 (\beta ^r z) = \frac{{\alpha ^r z}}{{1 - \alpha ^r z}} + \frac{{\beta ^r z}}{{1 - \beta ^r z}} = \frac{{L_r z - ( - 1)^r 2z^2 }}{{1 - L_r z + ( - 1)^r z^2 }}
\]
and
\[
\Li_0 (\alpha ^r z) - \Li_0 (\beta ^r z) = \frac{{\alpha ^r z}}{{1 - \alpha ^r z}} - \frac{{\beta ^r z}}{{1 - \beta ^r z}} = \frac{{F_r z\sqrt 5 }}{{1 - L_r z + ( - 1)^r z^2 }}
\]
in the identities of Theorem \ref{thm.main}, obtaining
\[
\sum_{j = 1}^\infty  {z^j F_{rj + s} }  = \frac{{F_s }}{2}\left( {\frac{{L_r z - ( - 1)^r 2z^2 }}{{1 - L_r z + ( - 1)^r z^2 }}} \right) + \frac{{L_s }}{{2\sqrt 5 }}\left( {\frac{{F_r z\sqrt 5 }}{{1 - L_r z + ( - 1)^r z^2 }}} \right)
\]
and
\[
\sum_{j = 1}^\infty  {z^j L_{rj + s} }  = \frac{{L_s }}{2}\left( {\frac{{L_r z - ( - 1)^r 2z^2 }}{{1 - L_r z + ( - 1)^r z^2 }}} \right) + \frac{{F_s \sqrt 5 }}{2}\left( {\frac{{F_r z\sqrt 5 }}{{1 - L_r z + ( - 1)^r z^2 }}} \right)\,,
\]
and hence identities \eqref{eq.ruu9ohg} and \eqref{eq.g74w6m7}. We used
\[
F_r L_s  + F_s L_r  = 2F_{r + s} \quad\mbox{Vajda~\cite[(16a)]{vajda}}\,,
\]
and
\[
L_r L_s  + 5F_r F_s  = 2L_{r + s},\quad\mbox{Vajda~\cite[(17a)+(17b)]{vajda}}\,. 
\]
\end{proof}
Identities \eqref{eq.ruu9ohg} and \eqref{eq.g74w6m7} are the generating functions of Fibonacci numbers and Lucas numbers with indices in arithmetic progression. To make this obvious, divide both sides of each identity by $z$, shift the index of summation by writing $j+1$ for $j$ and adjust the parameter $s$ by writing $s-r$ for $s$, thereby obtaining
\begin{equation}
\sum\limits_{j = 0}^\infty  {z^j F_{rj + s} }  = \frac{{F_s  - ( - 1)^r zF_{s - r} }}{{1 - L_r z + ( - 1)^r z^2 }}
\end{equation}
and
\begin{equation}
\sum\limits_{j = 0}^\infty  {z^j L_{rj + s} }  = \frac{{L_s  - ( - 1)^r zL_{s - r} }}{{1 - L_r z + ( - 1)^r z^2 }}\,.
\end{equation}
\begin{theorem}\label{thm.log}
Let $r$ and $s$ be integers. Let $z$ be a real or complex variable such that $|z|<\alpha^{-r}$. Then,
\begin{equation}
\sum_{j = 1}^\infty  {\frac{{z^{j} }}{j}F_{rj + s} }  =   -\frac{{F_s }}{2}\log \left( {1 - L_r z + ( - 1)^r z^2 } \right) - \frac{{L_s }}{{2\sqrt 5 }}\log \frac{{1 - \alpha ^r z}}{{1 - \beta ^r z}}\,,
\end{equation}

\begin{equation}
\sum_{j = 1}^\infty  {\frac{{z^{j} }}{j}L_{rj + s} }  =   -\frac{{L_s }}{2}\log \left( {1 - L_r z + ( - 1)^r z^2 } \right) - \frac{{F_s \sqrt 5 }}{2}\log \frac{{1 - \alpha ^r z}}{{1 - \beta ^r z}}\,.
\end{equation}
\end{theorem}
\begin{proof}
Use
\[
\Li_1 (\alpha ^r z) + \Li_2 (\beta ^r z) =  - \log \left( {1 - L_r z + ( - 1)^r z^2 } \right)
\]
and
\[
\Li_1 (\alpha ^r z) - \Li_2 (\beta ^r z) = \log \left( {\frac{{1 - \beta ^r z}}{{1 - \alpha ^r z}}} \right)
\]
in the identities of Theorem~\ref{thm.main}. 
\end{proof}
\begin{corollary}
Let $r$ and $s$ be integers. Let $x$ and $z$ be real variables such that $|z|<1$. Then,
\begin{equation}
\begin{split}
\sum_{j = 1}^\infty  {\frac{{z^j \cos jx}}{j}F_{rj + s} }  &=  - \frac{{F_s }}{4}\log \left( {z^4  - ( - 1)^r 2L_r z^3 \cos x + (L_{2r}  + ( - 1)^r 4\cos ^2 x)z^2  - 2L_r z\cos x + 1} \right)\\
&\qquad+ \frac{{L_s }}{{4\sqrt 5 }}\log \frac{{\beta ^{2r} z^2  - 2\beta ^r z\cos x + 1}}{{\alpha ^{2r} z^2  - 2\alpha ^r z\cos x + 1}}\,,
\end{split}
\end{equation}

\begin{equation}
\begin{split}
\sum_{j = 1}^\infty  {\frac{{z^j \cos jx}}{j}L_{rj + s} }  &=  - \frac{{L_s }}{4}\log \left( {z^4  - ( - 1)^r 2L_r z^3 \cos x + (L_{2r}  + ( - 1)^r 4\cos ^2 x)z^2  - 2L_r z\cos x + 1} \right)\\
&\qquad+ \frac{{F_s\sqrt 5 }}{{4}}\log \frac{{\beta ^{2r} z^2  - 2\beta ^r z\cos x + 1}}{{\alpha ^{2r} z^2  - 2\alpha ^r z\cos x + 1}}\,,
\end{split}
\end{equation}

\begin{equation}
\begin{split}
\sum_{j = 1}^\infty  {\frac{{z^j \sin jx}}{j}F_{rj + s} }  &= \frac{{F_s }}{2}\tan ^{ - 1} \frac{{zL_r \sin x - ( - 1)^r z^2 \sin 2x}}{{1 - zL_r \cos x + ( - 1)^r z^2 \cos 2x}}\\
&\qquad+ \frac{{L_s }}{{2\sqrt 5 }}\tan ^{ - 1} \frac{{zF_r \sqrt 5 \sin x}}{{1 - L_r z\cos x + ( - 1)^r z^2 }}\,,
\end{split}
\end{equation}

\begin{equation}
\begin{split}
\sum_{j = 1}^\infty  {\frac{{z^j \sin jx}}{j}L_{rj + s} }  &= \frac{{L_s }}{2}\tan ^{ - 1} \frac{{zL_r \sin x - ( - 1)^r z^2 \sin 2x}}{{1 - zL_r \cos x + ( - 1)^r z^2 \cos 2x}}\\
&\qquad+ \frac{{F_s\sqrt 5 }}{{2}}\tan ^{ - 1} \frac{{zF_r \sqrt 5 \sin x}}{{1 - L_r z\cos x + ( - 1)^r z^2 }}\,.
\end{split}
\end{equation}
\end{corollary}
\begin{proof}
Write $ze^{ix}$ for $z$ in the identities of Theorem \ref{thm.main}, with $k=1$, and take real and imaginary parts, noting that:
\begin{equation}
\Re \Li_1 (ze^{ix} ) =  - \frac{1}{2}\log (1 - 2z\cos x + z^2 )
\end{equation}
and
\begin{equation}
\Im \Li_1 (ze^{ix} ) = \tan ^{ - 1} \left( {\frac{{z\sin x}}{{1 - z\cos x}}} \right)\,.
\end{equation}
\end{proof}
\begin{theorem}\label{thm.xzk5qh8}
Let $s$ be an integer. Then,
\begin{equation}\label{eq.jh3oe13}
\sum_{j = 1}^\infty  {\frac{{( - 1)^{j-1} }}{{j^2 }}F_{j + s} }  = F_s \log ^2 \alpha  + \frac{{\pi ^2 }}{{50}}L_s \sqrt 5\,,
\end{equation}

\begin{equation}\label{eq.r20gb7l}
\sum_{j = 1}^\infty  {\frac{{( - 1)^{j-1} }}{{j^2 }}L_{j + s} }  = L_s \log ^2 \alpha  + \frac{{\pi ^2 \sqrt 5 }}{{10}}F_s\,.
\end{equation}
\end{theorem}
\begin{proof}
Using Theorem \ref{thm.main}, the left hand sides of \eqref{eq.jh3oe13} and \eqref{eq.r20gb7l} can be expressed as
\begin{equation}\label{eq.bzhxtze}
\sum_{j = 1}^\infty  {\frac{{( - 1)^{j - 1} }}{{j^2 }}F_{j + s} }  = \frac{{F_s }}{2}\left( {\Li_2 ( - \alpha ) + \Li_2 ( - \beta )} \right) + \frac{{L_s }}{{2\sqrt 5 }}\left( {\Li_2 ( - \alpha ) - \Li_2 ( - \beta )} \right)\,,
\end{equation}

\begin{equation}\label{eq.ru5rwan}
\sum_{j = 1}^\infty  {\frac{{( - 1)^{j - 1} }}{{j^2 }}L_{j + s} }  = \frac{{L_s }}{2}\left( {\Li_2 ( - \alpha ) + \Li_2 ( - \beta )} \right) + \frac{{F_s \sqrt 5 }}{2}\left( {\Li_2 ( - \alpha ) - \Li_2 ( - \beta )} \right)\,.
\end{equation}
It turns out that $\Li_2 ( - \alpha )$ and $\Li_2 ( - \beta)$ are elementary; and therefore the sums can be expressed in terms of elementary constants. We could use the values from Lewin~\cite{lewin81} but his results contain some misprints and so we have chosen to derive the needed values from scratch, as follows.

\medskip
 
Setting $x=\beta^2$ in the dilogarithm functional equation
\begin{equation}\label{eq.dr5tsju}
\Li_2 (x) + \Li_2 \left( {\frac{x}{{x - 1}}} \right) =  - \frac{1}{2}\log ^2 (1 - x)\,,
\end{equation}
gives
\begin{equation}\label{eq.hmfjin6}
\Li_2 (\beta ) + \Li_2 \left( {\beta ^2 } \right) =  - \frac{1}{2}\log ^2 \alpha\,.
\end{equation}
Use of $x=-\beta$ in the functional equation
\begin{equation}\label{eq.gv6ljpd}
\Li_2 \left( {\frac{1}{{1 + x}}} \right) - \Li_2 ( - x) = \frac{{\pi ^2 }}{6} - \frac{1}{2}\log (1 + x)\log \left( {\frac{{1 + x}}{{x^2 }}} \right),\quad x>0\,,
\end{equation}
and $x=\beta$ in the dilogarithm duplication formula
\begin{equation}\label{eq.qblycs8}
\Li_2 (x) + \Li_2 ( - x) = \frac{1}{2}\Li_2 (x^2 )
\end{equation}
gives
\begin{equation}\label{eq.vqzpvm4}
\Li_2 ( - \beta ) - \Li_2 (\beta ) = \frac{{\pi ^2 }}{6} - \frac{3}{2}\log ^2 \alpha
\end{equation}
and
\begin{equation}\label{eq.izvhw3t}
2\Li_2 ( - \beta ) + 2\Li_2 (\beta ) - \Li_2 (\beta ^2 ) = 0\,.
\end{equation}
Solving \eqref{eq.hmfjin6}, \eqref{eq.vqzpvm4} and \eqref{eq.izvhw3t} simultaneously for $\Li_2 ( - \beta )$, $\Li_2 ( \beta )$ and $\Li_2 ( \beta ^2 )$, we find
\begin{equation}\label{eq.ikpn2e1}
\Li_2 ( - \beta ) = \frac{{\pi ^2 }}{{10}} - \log ^2 \alpha\,, 
\end{equation}

\begin{equation}\label{eq.v6t8cqv}
\Li_2 (\beta ) =  - \frac{{\pi ^2 }}{{15}} + \frac{1}{2}\log ^2 \alpha \,,
\end{equation}
and
\begin{equation}\label{eq.mzu3bnd}
\Li_2 (\beta ^2 ) = \frac{{\pi ^2 }}{{15}} - \log ^2 \alpha\,. 
\end{equation}
Putting $x=-\alpha$ in \eqref{eq.dr5tsju}, noting \eqref{eq.ikpn2e1}, gives 
\begin{equation}\label{eq.x9vx7ia}
\Li_2 ( - \alpha ) =  - \frac{{\pi ^2 }}{{10}} - \log ^2 \alpha\,. 
\end{equation}
Thus,
\begin{equation}\label{eq.pryrigg}
\Li_2 ( - \alpha ) + \Li_2 ( - \beta ) =  - 2\log ^2 \alpha\,,
\end{equation}
and
\begin{equation}\label{eq.vumf2wl}
\Li_2 ( - \alpha ) - \Li_2 ( - \beta ) =  - \frac{{\pi ^2 }}{5}\,.
\end{equation}
Identities \eqref{eq.jh3oe13} and \eqref{eq.r20gb7l} now follow once we plug in \eqref{eq.pryrigg} and \eqref{eq.vumf2wl} in \eqref{eq.bzhxtze} and \eqref{eq.ru5rwan}.

\medskip

It may be noted that identity \eqref{eq.pryrigg} could have been obtained directly by setting $x=-\alpha$ in \eqref{eq.dr5tsju}.
\end{proof}
\begin{theorem}
We have
\begin{equation}\label{eq.vw9jb5e}
\sum_{j = 1}^\infty  {\frac{{( - 1)^{j - 1} }}{{j^2 }}L_{2j} }  = \frac{{\pi ^2 }}{6} + 2\log ^2 \alpha\,,
\end{equation}

\begin{equation}\label{eq.seqc5i4}
\sum_{j = 1}^\infty  {\frac{{( - 1)^{j - 1} }}{{j^2 }}L_{3j} }  = \frac{{\pi ^2 }}{{12}} + 6\log ^2 \alpha\,.
\end{equation}
\end{theorem}
\begin{proof}
Identity~\eqref{eq.partluca} with $k=2$, $z=-1$ and $r=2$ and $r=3$, in turn, gives
\begin{equation}\label{eq.mk0roai}
\sum_{j = 1}^\infty  {\frac{{( - 1)^{j - 1} }}{{j^2 }}L_{2j} }  = \Li_2 ( - \alpha ^2 ) + \Li_2 ( - \beta ^2 )
\end{equation}
and
\begin{equation}\label{eq.bstx15o}
\sum_{j = 1}^\infty  {\frac{{( - 1)^{j - 1} }}{{j^2 }}L_{3j} }  = \Li_2 ( - \alpha ^3 ) + \Li_2 ( - \beta ^3 )\,.
\end{equation}
Setting $x=\alpha^2$ in the dilogarithm inversion formula
\begin{equation}\label{eq.kyxha25}
\Li_2 ( - x) + \Li_2 ( - 1/x) =  - \pi ^2 /6 - \frac{1}{2}\log ^2 x,\quad x>0\,,
\end{equation}
produces
\begin{equation}\label{eq.vwcanf8}
\Li_2 ( - \alpha ^2 ) + \Li_2 ( - \beta ^2 ) =  - \pi ^2 /6 - 2\log ^2 \alpha\,,
\end{equation}
which, when plugged in \eqref{eq.mk0roai} gives \eqref{eq.vw9jb5e} while the use of $x=-\beta^3$ and $y=-\alpha^3$ in the functional equation
\begin{equation}\label{eq.zmv5iij}
\begin{split}
\Li_2 (xy) = \Li_2 (x) + & \Li_2 (y) - \Li_2 \left( {\frac{{x(1 - y)}}{{1 - xy}}} \right) - \Li_2 \left( {\frac{{y(1 - x)}}{{1 - xy}}} \right)\\
& - \log \left( {\frac{{1 - x}}{{1 - xy}}} \right)\log \left( {\frac{{1 - y}}{{1 - xy}}} \right)
\end{split}
\end{equation}
gives
\begin{equation}\label{eq.pl9zik3}
\Li_2 ( - \alpha ^3 ) + \Li_2 ( - \beta ^3 ) =  - \frac{{\pi ^2 }}{{12}} - 6\log ^2 \alpha\,,
\end{equation}
which together with \eqref{eq.bstx15o} yields identity~\eqref{eq.seqc5i4}.
\end{proof}
\begin{theorem}\label{thm.vpcaqaj}
We have
\[
\sum_{j = 1}^\infty  {\frac{{L_j }}{{2^j j^2 }}}  = \frac{{\pi ^2 }}{{12}} + 2\log ^2 \alpha  - \log ^2 2\,,
\]

\end{theorem}
\begin{proof}
Identity~\eqref{eq.partluca} with $k=2$, $z=1/2$ and $r=1$ gives
\begin{equation}\label{eq.vpcaqaj}
\sum_{j = 1}^\infty  {\frac1{{2^j j^2 }}L_j }  = \Li_2 (  \alpha /2) + \Li_2 (  \beta /2)\,.
\end{equation}
By using $x=\alpha/2$ and $y=\beta/2$ in the two-variable functional equation
\begin{equation}\label{eq.q1sdspz}
\begin{split}
\Li_2 \left( {\frac{x}{{1 - x}}\cdot\frac{y}{{1 - y}}} \right) = \Li_2 \left( {\frac{x}{{1 - y}}} \right)& + \Li_2 \left( {\frac{y}{{1 - x}}} \right) - \Li_2 (x) - \Li_2 (y)\\
& - \log (1 - x)\log (1 - y)\,,
\end{split}
\end{equation}
we find
\begin{equation}\label{eq.pzuyyyx}
\Li_2 (\alpha /2) + \Li_2 (\beta /2) = \frac{{\pi ^2 }}{{12}} + 2\log ^2 \alpha  - \log^22\,,
\end{equation}
which, plugged in \eqref{eq.vpcaqaj} gives the identity of the theorem.
\end{proof}
\begin{theorem}\label{thm.j7p1nrz}
Let $r$ be an even integer. Then,
\[
\sum_{j = 1}^\infty  {\frac{{L_{rj} }}{{L_r^j j^2 }}}  = \frac{{\pi ^2 }}{6} + r^2\log^2\alpha - \log^2L_r\,.
\]

\end{theorem}
\begin{proof}
Set $x=\alpha^r/L_r$ in the dilogarithm reflection formula
\begin{equation}\label{eq.ekmumc0}
\Li_2 (x) + \Li_2 (1 - x) = \frac{{\pi ^2 }}{6} - \log x\log (1 - x)\,,
\end{equation}
to obtain
\[
\Li_2 (\alpha ^r /L_r ) + \Li_2 (\beta ^r /L_r ) = \frac{{\pi ^2 }}{6} - \log (\alpha ^r /L_r )\log (\beta ^r /L_r )\,.
\]
The identity now follows from \eqref{eq.partluca} upon setting $k=2$, $z=1/L_r$.
\end{proof}
\begin{theorem}\label{thm.ejojlyt}
We have
\[
\sum_{j = 1}^\infty  {\frac{{( - 1)^{j - 1} }}{{j^3 }}L_j }  = \frac{1}{5}\left( {\pi ^2 \log \alpha  - \zeta (3)} \right)\,.
\]

\end{theorem}
\begin{proof}
Setting $k=3$, $z=-1$ in identity \eqref{eq.partluca} gives
\begin{equation}\label{eq.zvgs4vi}
\sum_{j = 1}^\infty  {\frac{{( - 1)^{j - 1} }}{{j^3 }}L_j }  = \Li_3 ( - \alpha ) + \Li_3 ( - \beta )\,.
\end{equation}
It is known that (see Lewin \cite[Formula 6.13]{lewin81} for a derivation):
\begin{equation}\label{eq.nqpxwc0}
\Li_3 (\beta ^2 ) = \frac{4}{5}\zeta (3) - \frac{{2\pi ^2 }}{{15}}\log\alpha + \frac{2}{3}\log ^3\alpha\,.
\end{equation}
Setting $x=\beta$ in the trilogarithm duplication formula
\begin{equation}\label{eq.v4uhxet}
\Li_3 (x) + \Li_3 ( - x)=\frac{1}{4}\Li_3 (x^2 )\,,
\end{equation}
gives
\begin{equation}\label{eq.rvva13a}
\Li_3 (\beta ) + \Li_3 ( - \beta ) = \frac{1}{4}\Li_3 (\beta ^2 )\,.
\end{equation}
Setting $x=-\beta$ in the trilogarithm inversion formula
\begin{equation}\label{eq.p0rydoq}
\Li_3 ( - x) - \Li_3 ( - 1/x) =  - \frac{{\pi ^2 }}{6}\log x - \frac{1}{6}\log ^3 x
\end{equation}
gives
\begin{equation}\label{eq.mvj6bil}
\Li_3 (\beta ) - \Li_3 ( - \alpha ) =  \frac{{\pi ^2 }}{6}\log\alpha + \frac{1}{6}\log ^3\alpha\,.
\end{equation}
Subtraction of \eqref{eq.mvj6bil} from \eqref{eq.rvva13a} produces
\begin{equation}\label{eq.agpln7g}
\Li_3 ( - \alpha ) + \Li_3 ( - \beta ) = \frac{1}{5}\zeta (3) - \frac{{\pi ^2 }}{5}\log\alpha\,,
\end{equation}
which, with \eqref{eq.zvgs4vi}, gives the identity of the theorem.

\end{proof}

\hrule

\noindent 2010 {\it Mathematics Subject Classification}:
Primary 11B39; Secondary 11B37.

\noindent \emph{Keywords: }
Fibonacci number, Lucas number, summation identity, series, generating function, polylogarithm, bernoulli polynomial, bernoulli number.

\hrule

\noindent Concerned with sequences: A000032, A000045

\hrule

\end{document}